\documentclass[12pt]{article}
\usepackage{amsmath,amsfonts,amssymb,graphicx,url,a4}

\usepackage{color}

\def\sach{\ |\ }
\def\passe{\lhd}

\renewcommand\leq{\le}
\renewcommand\geq{\ge}

\def\d{\mbox{\,{\rm d}}}
\def\R{{\mathbb R}}

\def\z{{\mathbb Z}}
\def\Z{{\mathbb Z}}

\def\n{{\mathbb N}}

\def\P{{\mathbb P}}
\def\E{\mathbb{E}}
\def\1{{\mathbb 1}}
\def\F{{\mathcal{F}}}

\newtheorem{theoreme}{Theorem}
\newtheorem{conjecture}{Conjecture}[section]
\newtheorem{lemme}[conjecture]{Lemma}
\newtheorem{proposition}[conjecture]{Proposition}

\newtheorem{remarque}{Remark}

\newtheorem{definition}{Definition}
\newtheorem{notation}{Notation}
\newtheorem{assumption}{Assumption}

\makeatletter
\def\@yproof[#1]{\@proof{ #1}}
\def\@proof#1{\begin{trivlist}\item[]{\em Proof#1.}}
\newenvironment{proof}{\@ifnextchar[{\@yproof}{\@proof{}
}}{~\end{trivlist}} \makeatother

\title{Sufficient conditions for the filtration of a stationary processes to be standard}
\author{Ga\"el Ceillier and Christophe Leuridan}
\date{\today}

\renewcommand{\le}{\leqslant}
\renewcommand{\ge}{\geqslant}

\begin{document}

\maketitle



\begin{abstract}
Let $X$ be a stationary process with values in some $\sigma$-finite measured state space $(E,\mathcal{E},\pi)$, indexed by ${\mathbb Z}$. Call ${\mathcal F}^X$ its natural filtration. In~\cite{Ceillier-stationary}, sufficient conditions were given for ${\mathcal F}^X$ to be standard when $E$ is finite, and the proof used a coupling of all probabilities on the finite set $E$.

In this paper, we construct a coupling of all laws having a density with regard to $\pi$, which is much more involved. Then, we provide sufficient conditions for ${\mathcal F}^X$ to be standard, generalizing those in~\cite{Ceillier-stationary}.
\end{abstract}

\begin{flushleft}
\end{flushleft}

\section{Introduction}

\subsection{Setting}

In this paper we focus on the natural filtration of 
stationary processes indexed by the integer line $\z$ with values in 
any $\sigma$-finite measured space. By definition,  
the natural filtration of a process $X=(X_n)_{n\in\z}$ is the 
nondecreasing sequence $\F^X=(\F^X_n)_{n\in\z}$ defined by
$\F^X_n=\sigma(X^{\passe}_n)$, where $X^{\passe}_n=(X_k)_{k \le n}$ denotes the 
$\sigma-$field generated by the ``past of $X$ at time $n$''. Furthermore, 
we set 
$\F^X_\infty=\sigma(X_k ; k\in\z)$ and $\F^X_{-\infty}=\bigcap_{k\in\z}\F^X_k$.
All the $\sigma-$fields that we consider here are assumed to be complete.

Let $X=(X_n)_{n\in\z}$ be any process defined on some 
probability space. Under loose separability conditions, and up to an 
enlargement of the 
probability space, one may assume that $X$ fulfills a recursion as follows: 
for every $n \in \z$, $X_{n+1}$ is a function of $X^{\passe}_{n}$ 
(the ``past'' of $X$ at time $n$) and of a ``fresh'' random variable 
$U_{n+1}$, which brings
in some ``new'' randomness. When this property holds, namely if, for every
$n \in \z$, (i) $U_{n+1}$ is independent of $\F^{X,U}_n$, and (ii)
$X_{n+1}$ is measurable with respect to $\sigma(U_{n+1})\vee\F^X_n$, 
we say that the process $U=(U_n)_{n\in\z}$ is a \textit{ parametrization \/} 
of $X$ (or of $\F^X$). In particular, the process $U=(U_n)_{n\in\z}$ must 
be independent. Although
Schachermayer's definition~\cite{Schachermayer DFST} of parametrizations imposes that the $U_n$ are uniform on $[0,1]$, we do not keep that restriction. 

Any process indexed by the integer line $\z$ 
with values in some essentially separable space possesses a parametrization 
(see a proof in~\cite{Ceillier-thesis}, section 5.0.5).

Likewise, we say that a parametrization $U$ is a \textit{generating parametrization\/} of $X$, or that $U$ \textit{generates\/} $X$ if, for every $n \in \z$, $X_n$ is measurable with respect to $\F^U_n$. This is equivalent to the condition that
  $\F^X_n \subset \F^U_n$ for every $n \in \z$, a property which, from now on, we write as $\F^X \subset \F^U$.

As a matter of fact, generating parametrizations provide a 
stronger property which is \textit{immersion\/} of filtrations.
Recall that the filtration $\F^X$ is \textit{immersed\/} in the filtration $\F^U$ if $\F^X\subset\F^U$ and if, for every $n \in\z$, $\F^X_{n+1}$ and $\F^U_n$ are independent conditionally on $\F^X_n$. Roughly speaking, this means that $\F^U_n$ gives no further information on $X_{n+1}$ than $\F^X_n$ does.
Equivalently, $\F^X$ is immersed in $\F^U$ if every $\F^X$-martingale is an $\F^U$-martingale. The following easy fact holds (see a proof in~\cite{Ceillier-thesis}).

\begin{lemme}
\label{immersion}
If $U$ is a generating parametrization of $X$, then $\F^X$ is immersed in $\F^U$.
\end{lemme}

Another notable property of filtrations is \textit{standardness\/}. Recall that $\F^X$ is \textit{standard\/}
if, modulo an enlargement of the probability space, one can immerse $\F^X$ in a filtration generated by an i.i.d.\ process with values in a separable space. Vershik introduced standardness in the context of ergodic theory.
Examples of non-standard filtrations include the filtrations of $[T,T^{-1}]$ transformations, introduced in~\cite{Heicklen - Hoffman}. Split-word processes, inspired by Vershik's $(r_n)$-adic sequences of decreasing partitions~\cite{Vershik 1995} and studied in~\cite{Smorodinsky 1998} and \cite{Laurent-thesis}, for instance, also provide non-standard filtrations.

Obviously, lemma~\ref{immersion} above implies that if $X$ has a generating parametrization with values in some essentially separable space, then $\F^X$ is standard. Whether the converse holds is not known.

Necessary and sufficient conditions for standardness include Vershik's
self-joining criterion and Tsirelson's notion of $I$-cosiness. Both
notions are discussed in~\cite{Emery - Schachermayer} and are based on 
conditions which are subtle and not easy to use nor to check in 
specific cases.

Our goal in this paper is to provide sufficient conditions for the 
existence of a generating parametrization of $\F^X$ 
(and therefore of standardness) that are
easier to use than the ones mentioned above. Each of our conditions involves a
quantification of the influence of the distant past
of the process on its present. We introduce them in the next section.

\subsection{Measuring the influence of the distant past}\label{intro-result}

From now on, $X$ denotes a stationary process indexed 
by the integer line $\z$ with values in some measurable space 
$(E,\mathcal{E})$. We fix a reference measure $\pi$ on $E$ which is 
$\sigma$-finite.
To introduce some quantities measuring the influence of the past of a 
process on its present, we need some notations.

\begin{notation}

(1) Slabs:
  For any sequence $(\xi_n)_{n\in\z}$ in $E^\z$, deterministic or random, and any integers $i\le
  j$, $\xi_{i:j}$ is the $(j-i+1)$-uple $(\xi_n)_{i\le n\le j}$ in
  $E^{j-i+1}$.
\\(2) Shifts:
If $k-i=\ell-j$, $\xi_{i:k}=\zeta_{j:\ell}$ means
  that $\xi_{i+n}=\zeta_{j+n}$ for every integer $n$ such that $0\le
  n\le k-i$.
\\Infinite case:  Let $E^{\passe}$ denote the space of sequences
  $(\xi_n)_{n\le-1}$.
For every $i$ in $\z$, a sequence $(\xi_n)_{n\le i}$ is also considered as an element of $E^{\passe}$ since, similarly to the finite case, one identifies $\xi^{\passe}_{i}=(\xi_n)_{n\le i}$ and $\zeta^{\passe}_{j}=(\zeta_n)_{n\le j}$ if $\xi_{i+n}=\zeta_{j+n}$ for every integer
  $n\le0$.
\\(3) Concatenation: for all $i \ge 0$, $j \ge 0$, $ x=(x_n)_{1\le n\le i}$ in $E^i$ and $y=(y_n)_{1\le n\le j}$ in $E^j$,
$xy$ denotes the concatenation of $x$ and $y$, defined as
$$
xy=(x_{1},\ldots,,x_{i},y_1,\ldots,y_j),\quad xy\in E^{i+j}.
$$
Infinite case: for all $j\ge 0$, $y=(y_n)_{1\le n\le j}$ in $E^j$ and
$x=(x_n)_{n\le-1}$ in $E^{\passe}$, $xy$ denotes the
concatenation of $x$ and $y$, defined as
$$
xy=(\ldots,x_{-2},x_{-1},y_1,\ldots,y_j),\quad xy\in E^{\passe}.
$$
\end{notation}

\begin{assumption}
 From now on, assume that one can choose a regular version of the conditional law of $X_0$ given $X_{-1}^{\passe}$ in such a way that for every $x \in E^\passe$, the law $\mathcal{L}(X_0|X_{-1}^{\passe}=x)$ has a density $f(\cdot|x)$ with respect to the reference $\sigma$-finite measure $\pi$ on $E$. Then, for every $n \ge 0$ and for every $x_{1:n} \in E^n$, the law $\mathcal{L}(X_0|X_{-n:-1}=x_{1:n})$ has a density $f(\cdot|x_{1:n})$ with respect to $\pi$.
\end{assumption}

When $E$ is countable and $\pi$ is the counting measure on $E$, the densities $f(\cdot|x)$ coincide with the functions $p(\cdot|x)$ used in~\cite{Ceillier-stationary} and defined by
$$
p(a|x) =\P (X_0 = a\sach X_{-n:-1}=x), \text{ for } x \in E^n \text{ and } a\in E,
$$
$$
\text{ and } \quad \quad p(a|x)=\P (X_0 = a \sach X^{\passe}_{-1}=x), \text{ for } x \in E^\passe \text{ and } a\in E.
$$

We now introduce three quantities $\gamma_n$, $\alpha_n$ and $\delta_n$ measuring the pointwise influence at distance $n$.

\begin{definition}
For every $n\ge0$, let
$$\begin{array}{l}
\gamma_n=1-\inf\left\{\displaystyle\frac{f(a|xz)}{f(a|yz)}\ ;\ a \in E,\ x\in E^{\passe},\ y \in E^{\passe},\ z\in E^n,\ f(a|yz)>0  \right\},\\~\\\alpha_n=1-\displaystyle\inf_{z \in E^n} \int_{E} \inf\left\{f(a|yz) ; y \in E^{\passe}\right\}\d\pi(a),\\~\\\delta_n=\sup\big\{ \| f(\cdot|xz)-f(\cdot|yz) \|\ ;\ x\in E^{\passe},\ y\in E^{\passe},\ z\in E^n\big\},
\end{array}$$
where, for every densities $f$ and $g$ on $E$,
$$\|f-g\| = \frac12\int_{E} |f(x)-g(x)|\d\pi(x) = 
\int_{E} \left[f(x)-g(x)\right]_+\d\pi(x).$$
\end{definition}

Note that the values $\gamma_n$, $\alpha_n$ and $\delta_n$ depend on the choice of the regular version of the conditional law of $X_0$ given $X_{-1}^{\passe}$. Applying the theorems below requires small influences, so one has to work with a ``good" version of the conditional law. Yet, replacing $\pi$ by an equivalent measure do not modify the quantities $\gamma_n$, $\alpha_n$ and $\delta_n$.

The sequences $(\gamma_n)_{n\geq 0}$, $(\alpha_n)_{n\geq 0}$ and 
$(\delta_n)_{n\geq 0}$ are non-increasing, $[0,1]$-valued, and
$\delta_n \le \min(\gamma_n,\alpha_n)$ for 
every $n \geq 0$ (see the proof in~\cite{Ceillier-stationary}, section 5).

For every $[0,1]$-valued sequence $(\varepsilon_n)_{n\geq 0}$, we consider the condition
\begin{align}
\tag{$\mathcal{H}(\varepsilon)$}
\sum_{k=0}^{+\infty} \prod_{n=0}^{k}(1-\varepsilon_n)=+\infty.
\end{align}
For instance, $\mathcal{H}(\gamma)$ and $\mathcal{H}(2\delta/(1+\delta))$ are respectively
$$
\sum_{k=0}^{+\infty} \prod_{n=0}^{k}(1-\gamma_n)=+\infty,
\quad\mbox{and}\quad
\sum_{k=0}^{+\infty} \prod_{n=0}^{k}\frac{1-\delta_n}{1+\delta_n}=+\infty.
$$

Observe that if two $[0,1]$-valued sequences $(\varepsilon_n)_{n\geq 0}$ and $(\zeta_n)_{n\geq 0}$ are such that $\varepsilon_n \leq \zeta_n$ for every $n\geq 0$, then $\mathcal{H}(\zeta)$ implies $\mathcal{H}(\varepsilon)$. Hence condition $\mathcal{H}(\varepsilon)$ asserts that $(\varepsilon_n)_{n\geq 0}$ is small in a certain way.

\subsection{Statement of the main results}

The definition of $(\gamma_n)_{n\geq 0}$ and the assumption $\mathcal{H}(\gamma)$ are both in \cite{Bressaud - Maass - Martinez - San Martin}.
The main result of \cite{Bressaud - Maass - Martinez - San Martin} is that when $E$ if of size $2$, then $\mathcal{H}(\gamma)$ implies that
$\F^X$ admits a generating parametrization.
This result is restricted by the following three conditions. First, the size of $E$ must be $2$. Second, one must control the \textit{ratios\/} of probabilities which define
$\gamma_n$. Third,
$\mathcal{H}(\gamma)$ implies that $\gamma_0 < 1$,
therefore one can show that $\mathcal{H}(\gamma)$ implies the existence of $c>0$ such that $f(a|x)\ge c$ for every $x$ in $E^{\passe}$ and $a$ in $E$.

Theorem 2 in~\cite{Ceillier-stationary} improves on this and gets rid of the first two restrictions. And Theorem~\ref{theo delta} below inproves on Theorem 2 in~\cite{Ceillier-stationary}.

\begin{theoreme}\label{theo delta}
  (1) If $E$ is finite and if $\mathcal{H}(2\delta/(1+\delta))$ holds, then $\F^X$ admits a generating parametrization.\\  (2) If the size of $E$ is $2$ and if $\mathcal{H}(\delta)$ holds, then $\F^X$ admits a generating parame\-trization.
\end{theoreme}

Theorem~\ref{theo delta} brings two improvements on the first result in~\cite{Ceillier-stationary}: the assumption $\mathcal{H}(2\delta)$ is replaced by $\mathcal{H}(2\delta/(1+\delta))$ and the extra hypothesis $\delta_0<1/2$ is not required anymore.

Another measure of influence, based on the quantities $\alpha_n$, is introduced and used in~\cite{Comets - Fernandez - Ferrari} (actually the notation there is $a_n=1-\alpha_n$). The authors show that if $\mathcal{H}(\alpha)$ holds, there exists a perfect sampling algorithm for the process $X$, a result which implies that $\F^X$ admits a generating parametrization.

Our Theorem~\ref{theo delta}, the result in \cite{Bressaud - Maass - Martinez - San Martin} and the exact sampling algorithm of
\cite{Comets - Fernandez - Ferrari} all require an upper bound of some pointwise
influence sequence. The second theorem in~\cite{Ceillier-stationary} uses a less restrictive
hypothesis based on some average influences $\eta_n$, defined below.

\begin{definition}
  For every $n\ge0$, let $\eta_n$ denote the average influence at distance $n$, defined as
  $$
    \eta_n= \E \big[\|f(\cdot|Y_{-n:-1})-f(\cdot|X^{\passe}_{-n-1}Y_{-n:-1})\|\big],
  $$
  where $Y$ is an independent copy of $X$, and call $\mathcal{H}'(\eta)$ the condition
  \begin{align}
    \tag{$\mathcal{H'}(\eta)$}
    \sum_{k=0}^{+\infty} \eta_k<+\infty.
  \end{align}
\end{definition}

As before, the definition of $\eta_n$ depends on the choice of the regular version of the conditional law of $X_0$ given $X_{-1}^{\passe}$, but it does not depend on the choice of the reference measure $\pi$. We recall the result from~\cite{Ceillier-stationary}.

\begin{theoreme}[Theorem 3 of~\cite{Ceillier-stationary}]\label{theo e}
  Assume that $E$ is finite and that for every $a$ in $E$,
  $p(a|X^{\passe}_{-1})>0$ almost surely (priming condition). 
Then, $\mathcal{H}'(\eta)$
  implies that $\F^X$ admits a generating parametrization.
\end{theoreme}

Note that the sequence $(\eta_n)_{n \geq 0}$ is $[0,1]$-valued. If $\eta_n<1$ for every $n \leq 0$, then $\mathcal{H}'(\eta)$ implies $\mathcal{H}(\eta)$.
Yet, since $\eta_n\le \delta_n$ for every $n \geq 0$, the condition $\mathcal{H}'(\eta)$ cannot be compared to the conditions $\mathcal{H}(\delta)$ and $\mathcal{H}(2\delta/(1+\delta))$.

The main result of the present paper is the theorem below which extends theorem 3 of ~\cite{Ceillier-stationary} to any $\sigma$-finite measured space $(E,\mathcal{E},\pi)$.

\begin{theoreme}\label{theo new}
Let $(E,\mathcal{E},\pi)$ be any $\sigma$-finite measured space. 
Assume that one has chosen a regular version of the conditional law of 
$X_0$ given $X_{-1}^{\passe}$ in such a way that for every $x \in E^\passe$, 
the law $\mathcal{L}(X_0|X_{-1}^{\passe}=x)$ has a density $f(\cdot|x)$ 
with respect to $\pi$ and that for $\pi$-almost every $a$ in $E$,
$f(a|X^{\passe}_{-1})>0$ almost surely (priming condition). 
Then, $\mathcal{H}'(\eta)$
implies that $\F^X$ admits a generating parametrization, hence is standard.
\end{theoreme}

In this paper, we call {\it priming condition} the condition ``for $\pi$-almost every $a$ in $E$, $f(a|X^{\passe}_{-1})>0$ almost surely''. 
This generalises the priming condition of~\cite{Ceillier-stationary}. Indeed,
when $E$ is finite and $\pi$ is the counting measure on $E$, the priming condition is equivalent to the assumption that for every $a$ in $E$, $p(a|X^{\passe}_{-1})>0$ almost surely, that is, the condition involved in theorem~\ref{theo e}. An equivalent form, which is actually the one used in the proof of theorem~\ref{theo e}, is
  $$\inf_{a \in E} p(a|X^{\passe}_{-1})>0 \text{ almost surely.}$$
But this condition cannot be satisfied anymore when $E$ is infinite. Therefore, the proof of theorem~\ref{theo new} requires new ideas to use the priming condition, see our priming lemma in section~\ref{priming-lemma} (lemma~\ref{lemme amorce2}).


\subsection{Global couplings}

A key tool for the construction of our parametrization is a new global coupling. This global coupling, which applies when $E$ is infinite, is much more involved than the coupling introduced in~\cite{Ceillier-stationary} when $E$ is finite.

Recall that if $p$ and $q$ are two {\it fixed} probabilities on a countable space $E$, then for every random variables $Z_p$ and $Z_q$ with laws $p$ and $q$ defined on the same probability space,
$$
\P[Z_p\ne Z_q]\geq\|p-q\|,
$$
where $\|p-q\|$ is the
distance in total variation between $p$ and $q$, defined as
$$
\|p-q\|=\frac12\sum_{a \in E} |p(a)-q(a)|=\sum_{a \in E} \left[p(a)-q(a)\right]_+.
$$
Conversely, a standard construction in coupling theory provides some random variables $Z_p$ and $Z_q$ with laws $p$ and $q$ such that
$
\P[Z_p\ne Z_q]=\|p-q\|.
$

Couplings can be generalized to any set of laws, on a measurable space $(E,\mathcal{E})$.
\begin{definition}
Note $\mathcal{P}(E)$ the set of probabilities on the measurable space $(E,\mathcal{E})$. Let $\mathcal{A}$ be a subset of $\mathcal{P}(E)$. A global coupling for $\mathcal{A}$ is a random variable $U$ with values in some measurable space $(F,\mathcal{F})$ and a function $g$ from $F \times \mathcal{A}$ to $E$ such that for every $p \in \mathcal{A}$, $g(U,p)$ is a random variable of law $p$ on $E$.
\end{definition}

Good global couplings are such that for every $p$ and $q$ in $\mathcal{A}$, the probability
$\P[g(U,p)\ne g(U,q)]$ is small when $\|p-q\|$ is small.

When $E=\{0,1\}$, a classical coupling used in~\cite{Bressaud - Maass - Martinez - San Martin} is given by $g(U,p)=\mathbf{1}_{\{U\leq p(1)\}},$ where $U$ is a uniform random variable on $[0,1]$. This coupling satisfies the equality $\P[g(U,p)\ne g(U,q)]=\|p-q\|$, for every $p$ and $q$ in $\mathcal{P}(E)$.

This coupling has been extended to any finite set in~\cite{Ceillier-stationary}. Furthermore, the construction can still be extended to any countable set, as follows.
\begin{proposition}\label{global-countable}
Assume that $E$ is countable. Let $\varepsilon=(\varepsilon_a)_{a\in E}$ be an i.i.d. family of exponential random variables with parameter 1. Then, for any probability $p$ on $E$, there exists almost surely one and only one $a\in E$ such that
$$\frac{\varepsilon_a}{p(a)}=\inf_{b \in E} \frac{\varepsilon_b}{p(b)} \quad (\star),$$
with the convention that $\varepsilon_b/p(b)=+\infty$ if $p(b)=0$.
\begin{itemize}
\item Define $g(\varepsilon,p)$ (almost surely) as the only index $a$ such that $(\star)$ holds. Then the law of the random variable $g(\varepsilon,p)$ is $p$.
\item For every $p$, $q$ in $\mathcal{P}(E)$,
$$
\P[g(\varepsilon,p) \ne g(\varepsilon,q)] \leq 2\frac{\|p-q\|}{1+\|p-q\|} \leq  2\, \|p-q\|.
$$
\end{itemize}
\end{proposition}

The bound $2\|p-q\|/(1+\|p-q\|)$ improves on the bound $2\|p-q\|$ given in proposition 2.6 of~\cite{Ceillier-stationary} and explains that the assumption $\mathcal{H}(2\delta/(1+\delta))$
of theorem~\ref{theo delta} is weaker than the assumptions $\mathcal{H}(2\delta)$ and $\delta_0<1/2$ required in theorem 2 of~\cite{Ceillier-stationary}, although the same proof works.
The proof of proposition~\ref{global-countable} will be given in section~\ref{S_aux}.



We now briefly introduce a new global coupling that works on any measured $\sigma$-finite space $(E,\mathcal{E},\pi)$. The construction and the properties of the coupling $(g,U)$ will be detailed in section~\ref{S_global}.

Let $\mathcal{P}_\pi(E)$ be the subset of $\mathcal{P}(E)$ of all probability laws $p$ on $E$ admitting a density $f_p$ with respect to $\pi$. Let $\lambda$ be the Lebesgue measure on $\R^+$.
Let $U$ be a Poisson point process on $E\times \R^+\times\R^+$ of intensity $\pi\otimes\lambda\otimes\lambda$.

In the following, any element $(x,y,t)$ of $U$ is seen as a point $(x,y)$ in $E\times\R^+$ appearing at the time $t$.
Since $\lambda$ is diffuse, the third components of the elements of $U$ are almost surely all distinct.
For any density function $f$ on $(E,\mathcal{E},\pi)$, set $D_f=\{(x,y)\in E\times\R^+ ~:~y\leq f(x)\}$ and call $t_f$ the instant of appearance of a point under the graph of $f$. Then, almost surely, $t_f$ is positive and finite and the set $U\cap (D_f \times [0,t_f])$ is almost surely reduced to a single point $(x_f,y_f,t_f)$. The random variable $x_f$ has density $f$ with respect to $\pi$.

\begin{proposition}\label{prop_up_bound}
With the notations above, the formula $g(U,p)=x_{f_p}$ defines a global coupling for $\mathcal{P}_\pi(E)$ such that for every $p$ and $q$ in $\mathcal{P}_\pi(E)$,
$$\P[g(U,p)\ne g(U,q)]\leq \frac{2 \|p-q\|}{1+\|p-q\|}\leq 2 \|p-q\|.$$
\end{proposition}

The paper is organised as follows : In section~\ref{S_global}, we detail the construction of the global coupling just mentionned, we prove proposition~\ref{prop_up_bound} and we construct a parametrization of the process $(X_n)_{n\in\z}$. Section~\ref{priming-lemma} is devoted to the proof of a key intermediate result which we call priming lemma, which uses the priming condition. In section~\ref{S_end of the proof}, we complete the proof of theorem~\ref{theo new}, showing that the parametrization $(U_n)_{n\in \z}$ constructed in section~\ref{S_global} generates the filtration $\F^X$.

\section{Construction of a parametrization with a glo-bal coupling}\label{S_global}
\subsection{Construction of a global coupling}

Let $(E,\mathcal{E},\pi)$ be a $\sigma-$finite measured space. Let $\lambda$ be the Lebesgue measure on $\R^+$. Set $\mu=\pi\otimes\lambda$ on $E\times \R^+$. Denote by $\mathcal{D}$ the set of countable subsets of $E\times \R^+\times\R^+$.

Let $U$ be a Poisson point process on $E\times \R^+\times\R^+$ of intensity $\pi\otimes\lambda\otimes\lambda$. By definition, $U$ is a $\mathcal{D}$-valued random variable such that for any pairwise disjoint sets $B_k$ in $\mathcal{E}$, of finite measure for $\mu\otimes\lambda=\pi\otimes\lambda\otimes\lambda$, the cardinalities $|U \cap B_k|$ are independent Poisson random variables with respective parameters $(\mu\otimes\lambda)(B_k)$.

In the following, any element $(x,y,t)$ of $U$ is seen as a point $(x,y)$ in $E\times\R^+$ appearing at time $t$.
Since $\lambda$ is diffuse, the third components of the elements of $U$ are almost surely all distinct.

\begin{notation}For any measurable subset $A$ of $E\times\R^+$, the time of appearance of a point in $A$ is
$$t_A(U)=\inf\left\{t\ge 0 ~:~ U\cap(A\times[0,t])\ne\emptyset \right\}.$$
\end{notation}

Let us recall some classical and useful properties of Poisson point processes. If the measure $\mu(A)$ is positive and finite, the random variable $t_A(U)$ has exponential law of parameter $\mu(A)$. Moreover the set $U\cap (A\times[0,t_A(U)])$ is almost surely reduced to a single point $(x_A(U),y_A(U),t_A(U))$. The couple $(x_A(U),y_A(U))$ thus defined is a random variable of law $\mu(\cdot|A)=\mu(\cdot\cap A)/\mu(A)$ which is independent from $t_A(U)$.

Moreover if the sets $B_k$ are pairwise disjoint measurable subsets of $E\times\R^+$, the random variables $(x_{B_k}(U),y_{B_k}(U),t_{B_k}(U))$ are independent.

The coordinates $x_A(U),y_A(U),t_A(U)$ vary as a function of the variable $U$. This notation will be used later on when we will consider several Poisson point processes, but, by abuse of notation it will be abbreviated to $x_A, y_A,t_A$ when there is no ambiguity on the Poisson point process.

\begin{lemme}\label{lemme1}
If the $n$ sets $A_k$ are measurable subsets of $E\times\R^+$, of finite positive measures for $\mu$, then almost surely
$$t_{A_1}=...=t_{A_n} \Longleftrightarrow t_{A_1 \cap...\cap A_n}=t_{A_1 \cup...\cup A_n}. $$
\end{lemme}
{\it Proof}
The reverse implication follows from the inequalities
$$t_{A_1 \cup...\cup A_n}=\min\{t_{A_1},...,t_{A_n}\}\leq \max\{ t_{A_1},...,t_{A_n}\}\leq t_{A_1 \cap...\cap A_n}.$$
The direct implication follows from the fact that, almost surely, the third components of the elements of $U$ are all distinct. Therefore if $t_{A_1}=...=t_{A_n}$, then the point $(x_{A_k},y_{A_k},t_{A_k})$ does not depend on $k$, therefore $(x_{A_k},y_{A_k})$ belongs to $A_1 \cap...\cap A_n$.
This ends the proof. \hfill $\square$

We now study the dependance of $x_A$ with respect to $A$.

\begin{proposition}\label{erreur}
Let $A$ and $B$ be two measurable subsets of $E\times\R^+$ of measures finite and positive, one has that
$$\{x_A=x_B\} \supset \{t_A=t_B\} \text{ almost surely}$$
and $$\P[t_A=t_B]=\frac{\mu(A\cap B)}{\mu(A\cup B)}.$$
\end{proposition}
{\it Proof}
Since the third components of the elements of $U$ are almost surely all distinct, one has $$t_A=t_B \Rightarrow x_A=x_B \text{ almost surely}.$$
By lemma~\ref{lemme1}, and the equality $t_{A \cup B}=\min \{t_{A\cap B},t_{A\triangle B}\}$,
\begin{eqnarray*}
    t_A=t_B &\Longleftrightarrow& t_{A\cap B}=t_{A \cup B}\\    &\Longleftrightarrow& t_{A\cap B}\leq t_{A\triangle B}.
\end{eqnarray*}
Since $t_{A\cap B}$ and $t_{A\triangle B}$ are independent exponential random variables with parameters $\mu(A\cap B)$ and $\mu(A\triangle B)$, one gets $$\P[t_A=t_B]=\mu(A\cap B)/\mu(A\cup B),$$
which completes the proof. \hfill $\square$

\begin{remarque}
Note that if $\pi$ is diffuse, then $\{x_A=x_B\} = \{t_A=t_B\}$ almost surely, thus $\P[x_A=x_B]=\mu(A\cap B)/\mu(A\cup B).$
\end{remarque}

Let us introduce some subsets to which we will apply proposition~\ref{erreur}.
\begin{notation}
For any measurable map $f : E \rightarrow \R^+$, denote by $D_f$ the part of $E\times\R^+$ located below the graph of $f$:
$$D_f=\{(x,y)\in E\times\R^+ ~:~y\leq f(x)\}.$$
\end{notation}
Then,$$\mu(D_f)=\int_E f(x)\d\pi (x).$$
If this measure is positive and finite, then we denote in abbreviated form :
$$(x_f,y_f,t_f)=(x_{D_f},y_{D_f},t_{D_f}).$$




Proposition~\ref{prop_up_bound} is a direct consequence of the next lemma.

\begin{lemme}\label{lemme-coupl-loi}
Let $\alpha>0$ and $f$ be a probability density function on $(E,\pi)$. Then
\begin{itemize}
\item the random variable $x_{\alpha f}$ has density $f$ with respect to $\pi$. Thus $(U,x)$ is a global coupling for the set of all laws on $E$ with a density with respect to $\pi$
\item for all probabilities $f$ and $g$ on $E$,
$$\P[x_{f}=x_{g}]\geq\P[t_{f}=t_{g}]=\frac{1-\|f-g\|}{1+\|f-g\|}.$$
\end{itemize}
\end{lemme}
{\it Proof}
The first result follows from the fact that the law of $(x_{\alpha f}, y_{\alpha f})$ is $\mu(\cdot|D_{f})$. The second result follows from proposition~\ref{erreur}. \hfill $\square$



\subsection{Parametrizing one random variable}\label{intro2}

Let $X$ be a random variable with values in $E$ with a density $f$ with respect to $\pi$.
Let $W$ be a Poisson point process on $E\times \R^+\times\R^+$ of intensity $\pi\otimes\lambda\otimes\lambda$, and $V$ be a random variable uniform on $[0,1]$, such that $X,V,W$ are independent. Let us modify $W$ in such a way that the first two components of the first point which appears in $D_f$ are $(X,Vf(X))$. We set
$$U=\Big[W\setminus \big\{\big(x_f(W),y_f(W),t_f(W)\big)\big\}\Big]\cup\big\{\big(X,Vf(X),t_f(W)\big)\big\}.$$

\begin{proposition}\label{construction-governing} The process $U$ thus defined is a Poisson point process on $E\times \R^+\times\R^+$ of intensity $\pi\otimes\lambda\otimes\lambda$, such that $x_f(U)=X$.
\end{proposition}
{\it Proof}
One has directly  $x_f(U)=X$. Let us show that $U$ is a Poisson point process.

Denote by $\overline{D_f}$ the complementary of $D_f$ in $E\times\R^+$. Let us split $W$ into two independent Poisson point processes $W_1$ and $W_2$ by setting :
$$W_1=W \cap (D_f\times\R^+) \text{ and } W_2=W \cap (\overline{D_f}\times\R^+). $$

In a similar way, set
$$U_1=U \cap (D_f\times\R^+) \text{ and } U_2=U \cap (\overline{D_f}\times\R^+). $$
It suffices to show that $\mathcal{L}(U_1,U_2)=\mathcal{L}(W_1,W_2)$. Note that $U_2=W_2$ and $U_1$ is a function of $X,V,W_1$, therefore $U_1$ is independent from $U_2$ and $\mathcal{L}(U_2)=\mathcal{L}(W_2)$. Thus what remains to be proved is that $\mathcal{L}(U_1)=\mathcal{L}(W_1)$.

Let $(w_k,t_k)_{k \geq 1}$ be the sequence of the points of $W_1$ in $D_f \times \R^+$ ordered in such a way that $(t_k)_{k \geq 1}$ is an increasing sequence. Then $$U_1=\bigcup_{k \geq 1} \{(u_k,t_k)\}$$ with $u_1=(X,Vf(X))$ and $u_k=w_k$ for $k \geq2$. The random variables $(w_k)_{k \geq 1}$ form an i.i.d. sequence of law $\mu(\cdot|D_f)$, independent of $((t_k)_{k \geq 1},X,V)$.

A direct computation shows that the random variable $u_1=(X,Vf(X))$ has law $\mu(\cdot|D_f)$. Moreover, $u_1$ is independent of $((w_k)_{k \geq 2},(t_k)_{k\geq1})$
by independence of $X,V$ and $W$. Therefore $\mathcal{L}(U_1)=\mathcal{L}(W_1)$, which completes the proof. \hfill $\square$

\subsection{Construction of a parametrization}

Assume that $(X_n)_{n \in \z}$ is a stationary process with values in $E$ such that given $X^\passe_{-1}$, $X_0$ admits conditional densities $(f(\cdot|x))_{x \in E^\passe}$ with respect to $\pi$. One wants to apply the previous construction at each time $n \in \z$.

Recall that for every $n \in \Z^+$ and $x \in E^{\Z^-_*}$ the function $f(\cdot|x)$ is the density of the law $\mathcal{L}(X_0|X^{\passe}=x)$ with respect to $\pi$.

Let $(W_n)_{n\in\Z}$ be a sequence of independent Poisson point processes on $E\times \R^+\times\R^+$ of intensity $\pi\otimes\lambda\otimes\lambda$, and let $(V_n)_{n\in\Z}$ be a sequence of independent and uniform random variables on $[0,1]$, such that the sequences $(X_n)_{n \in \z},(V_n)_{n \in \z},(W_n)_{n \in \z}$ are independent.
Set
$$U_n=W_n\cup\{(X_n,V_nf_{n-1}(X_n),t_{f_{n-1}}(W_n))\}\setminus \{(x_{f_{n-1}}(W_n),y_{f_{n-1}}(W_n),t_{f_{n-1}}(W_n))\},$$
where $f_{n-1}=f(\cdot|X^{\passe}_{n-1})$.
Since $X$ is stationary, $f_{n-1}$ is the law of $X_n$ given the past $X^{\passe}_{n-1}$. Proposition~\ref{construction-governing} implies that for every $n \in \z$, $U_n$ is a Poisson point process independent from $\F^{X,V,W}_{n-1}$. Moreover, $X_n=x_{f_{n-1}}(U_n)$ and $U_n$ is a function of $W_n$, $V_n$ and $X^{\passe}_{n}$. Therefore we have the following result.
\begin{proposition}\label{governing sequence}
The process $(U_n)_{n \in \z}$ thus defined is an i.i.d. sequence of Poisson point processes on $E\times \R^+\times\R^+$ of intensity $\pi\otimes\lambda\otimes\lambda$, which is a parametrization of $\F^X$ with the recursion formula $X_n=x_{f_{n-1}}(U_n)$.
\end{proposition}

\section{Priming lemma}\label{priming-lemma}

Fix an integer $\ell \ge 0$. The following lemma provides a random variable $Z_{1:\ell}$ and an event $H_\ell$ which both depend only on $U_{1:\ell}=(U_1,...,U_\ell)$ such that when $H_\ell$ occurs,
$X_{1:\ell}=Z_{1:\ell}$ with probability close to $1$.

\begin{lemme}[Priming lemma]\label{lemme amorce2}
    Assume that for \- $\pi$-almost every $a$ in $E$, we have $f(a|X^{\passe}_{-1})>0$ almost surely. For every $\varepsilon>0$ and $\ell \ge 0$, there exists a random variable $Z_{1:\ell}=(Z_1,...,Z_\ell)$ with values in $E^\ell$ and an event $H_\ell$, which are both functions of $U_{1:\ell}=(U_1,...,U_\ell)$ only, such that
    \begin{itemize}
        \item $\P[H_\ell]>0$,
        \item $\mathcal{L}(Z_{1:\ell})=\mathcal{L}(X_{1:\ell})$,
        \item $H_\ell$ is independent of $Z_{1:\ell}$,
        \item $\P[X_{1:\ell}\ne Z_{1:\ell}~|~H_\ell] \le \varepsilon.$
    \end{itemize}
\end{lemme}

Note that the hypothesis and the conclusions of lemma~\ref{lemme amorce2} are not modified if one replace $\pi$ by an equivalent measure. Thus in the proof, one may assume without loss of generality that $\pi$ is a probability measure.

The proof proceeds by induction. For $\ell=0$ the result is trivial with $H_0=\Omega$ and $Z_0$ equal to the empty word. The inductive step from $\ell$ to $\ell+1$ uses the next two lemmas. These lemmas show that with a probability close to $1$, most of the graph of $f(\cdot|X^\passe_\ell)$ is between $m^{-1}f(\cdot|Z_{1:\ell})$ and $n$ for suitable constants $m$ and $n$.

\begin{lemme}\label{lemme_amorce_1} Let $\ell \ge 0$ be an integer, $H$ an event of positive probability, and $Z$ be a random variable with values in $E^\ell$. For every $\varepsilon>0$, there exists a real number $m\ge1$ such that
$$\E\Big[\int_E \big[f(a|Z)-m f(a|X^\passe_\ell)\big]_+\d\pi(a)~\Big|~H\Big]\le \varepsilon.$$
\end{lemme}

\begin{proof} The priming condition and the stationarity of $X$ ensure that, 
for $\pi$-almost every $a \in E$, $f(a|X^{\passe}_{\ell})>0$ almost surely, so
$$\big[f(a|Z)-m f(a|X^\passe_\ell)\big]_+ \to 0 \text{ a.s. as }m \to +\infty.$$
Hence, by dominated convergence
$$\E\Big[\int_E \big[f(a|Z)-m\ f(a|X^\passe_\ell)\big]_+\d\pi(a)\Big] \to 0 \text{ as }m \to +\infty.$$
The same result holds with $\E[\cdot|H]$ instead of $\E$ since $\P[\cdot|H]\le\P[H]^{-1}\P[\cdot]$,
which concludes the proof.
\hfill $\square$
\end{proof}

\begin{lemme}\label{lemme_amorce_2} Let $\ell \ge 0$ be an integer, $m \ge 1$ a real number, $H$ an event of positive probability and $Z$ be a random variable with values in $E^\ell$. For every $\varepsilon>0$, there exists a real number $n\ge1$ such that
$$\E\Big[\int_E \big[f(a|X^\passe_\ell)-n\big]_+\d\pi(a)~\Big|~H\Big]\le \varepsilon.$$
For such a real number $n$, there exists a random variable $M$, with values in $[n,n+1]$, which is a function of $Z$ only, such that
$$\int_E \sup(m^{-1} f(x|Z) , M)\d\pi(x)=n+1.$$
Since $M \ge n$, one has $\big[f(a|X^\passe_\ell)-M\big]_+ \le \big[f(a|X^\passe_\ell)-n\big]_+$, so
$$\E\Big[\int_E \big[f(a|X^\passe_\ell)-M\big]_+\d\pi(a)~\Big|~H\Big]\le \varepsilon.$$
\end{lemme}

\begin{proof} The same method as in the proof of lemma~\ref{lemme_amorce_1} provides a real number $n\ge 1$ such that
$$\E\Big[\int_E \big[f(a|X^\passe_\ell)-n\big]_+\d\pi(a)~\Big|~H\Big]\le \varepsilon.$$
Define a random application $\phi$ from $\R^+$ to $\R^+$  by
$$\phi(s)= \int_E \sup(m^{-1} f(x|X^\passe_\ell) , s)\d\pi(x).$$
 Since for every $s \in \R^+$, $s\le\sup(m^{-1} f(x|X^\passe_\ell),s)\le m^{-1} f(x|X^\passe_\ell)+s$, one has $s\le\phi(s)\le m^{-1}+s\le 1+s$ (recall that $\pi$ is assumed to be a probability). Hence $\phi(n)\le n+1\le\phi(n+1)$. Then, since $\phi$ is continuous, the random variable
$$M=\inf\{s\in \R^+:\,\phi(s)= n+1\}$$ is well defined and satisfies the
conclusion of the lemma.
\hfill $\square$
\end{proof}

We can now prove the induction step of the proof of lemma~\ref{lemme amorce2}.

\begin{proof} Let $\varepsilon>0$ and $\ell \in \n$. 
Assume that one has constructed $H_\ell$ and $Z_{1:\ell}$, which 
are both functions of $U_{1:\ell}$ only, and such that
    \begin{itemize}
        \item $\P[H_\ell]>0$,
        \item $\mathcal{L}(Z_{1:\ell})=\mathcal{L}(X_{1:\ell})$,
        \item $H_\ell$ is independent of $Z_{1:\ell}$,
        \item $\P[Z_{1:\ell}\ne X_{1:\ell}~|~H_\ell] \le \varepsilon/3.$
    \end{itemize}
Lemmas~\ref{lemme_amorce_1} and~\ref{lemme_amorce_2}, applied to $\ell$, 
$H=H_\ell$, $Z=Z_{1:\ell}$ and $\varepsilon/3$ provide two real numbers 
$m \ge 1$, $n \ge 1$ and a $\sigma(Z_{1:\ell})$-measurable random variable $M$. 
Set 
$$f_\ell=f(\cdot|X^\passe_\ell),\quad f'_\ell=f(\cdot|Z_{1:\ell}),
\quad A=D_{m^{-1}f'_\ell},\quad B=D_M, \quad C=D_{f_\ell}.$$ 
The random sets $A$ and $B$ are $\sigma(Z_{1:\ell})$-measurable, 
therefore $\sigma(U_{1:\ell})$-measurable, whereas $C$ is 
$\F^X_{\ell}$-measurable. Moreover, by lemma~\ref{lemme_amorce_2},
$$\mu(A\cup B)=\int_E \sup(m^{-1}f(a|Z_{1:\ell}),M) \d\pi(a) = n+1,$$
whereas
$$\mu(A) = \int_E m^{-1}f(a|Z_{1:\ell}) \d\pi(a) = m^{-1}.$$
Let $$H' = \{t_A(U_{\ell+1})\le t_B(U_{\ell+1})\} = 
\{t_A(U_{\ell+1})\le t_{B \setminus A}(U_{\ell+1})\},$$
$$H_{\ell+1} = H_\ell \cap H', \quad Z_{\ell+1}=x_A(U_{\ell+1}),$$
and recall that $X_{\ell+1}=x_{f_\ell}(U_{\ell+1})=x_{C}(U_{\ell+1})$ 
by proposition~\ref{governing sequence}.



For every $F \in \mathcal{E}^{\otimes \ell}$ and $G \in \mathcal{E}$, 
\begin{equation}\label{conditionning1}
\P[Z_{1:\ell+1}\in F\times G~;~H_{\ell+1}~|~\F^{X,U}_{\ell}]
= {\bf 1}_F(Z_{1:\ell})~{\bf 1}_{H_\ell}~ 
\P[Z_{\ell+1}\in G~;~H'~|~\F^{X,U}_{\ell}].
\end{equation}
Conditionally on $\F^{X,U}_{\ell}$, the random variables $x_A(U_{\ell+1})$, 
$t_A(U_{\ell+1})$ and $t_{B \setminus A}(U_{\ell+1})$ are independent, with 
respective laws $f'_\ell \cdot \pi$, exponential with parameter 
$\mu(A)$, exponential with parameter $\mu(B \setminus A)$. Therefore, 
\begin{equation}\label{conditionning2}
\P[Z_{\ell+1}\in G~;~H'~|~\F^{X,U}_{\ell}] 
= \frac{\mu(A)}{\mu(A \cup B)} \int_G f'_\ell \d\pi 
= \frac{1}{m(n+1)} \int_G f'_\ell \d\pi.
\end{equation}
In particular, $\P[H'|\F^{X,U}_{\ell}]=\frac{1}{m(n+1)}$, 
so $H'$ is independent of $\F^{X,U}_{\ell}$; this fact 
will be used at the end of the proof. 

Combining equalities~\ref{conditionning1} and~\ref{conditionning2}, 
taking expectations and using the induction hypothesis yield 
\begin{eqnarray*}
\P[Z_{1:\ell+1}\in F\times G~;~H_{\ell+1}] 
&=& \frac{1}{m(n+1)}
\E\Big[{\bf 1}_F(Z_{1:\ell})~{\bf 1}_{H_\ell}~\int_Gf(\cdot|Z_{1:\ell})\d\pi\Big]\\
&=& \frac{\P(H_\ell)}{m(n+1)} 
\E\Big[{\bf 1}_F(X_{1:\ell})~\int_Gf(\cdot|X_{1:\ell})\d\pi\Big]\\
&=& \frac{\P(H_\ell)}{m(n+1)} \P[X_{1:\ell+1}\in F\times G].
\end{eqnarray*}
This shows that $Z_{1:\ell+1}$ and $H_{\ell+1}$ are independent, 
that $\mathcal{L}(Z_{1:\ell+1})=\mathcal{L}(X_{1:\ell+1})$ and 
that $\P(H_{\ell+1})=\P(H_\ell)/(m(n+1))>0$.

What remains to be proved is the upper bound of 
$\P[Z_{\ell+1}\ne X_{\ell+1}~|~H_{\ell+1}]$.
By proposition~\ref{erreur} and lemma~\ref{lemme1},
\begin{eqnarray*}
\big\{X_{\ell+1}=Z_{\ell+1} \big\}\cap H' &\supset& \big\{ t_C(U_{\ell+1})=t_A(U_{\ell+1})=t_{A\cup B}(U_{\ell+1})\big\}\\
&=&\big\{t_{A\cap C \cap (A \cup B)}(U_{\ell+1})=t_{A\cup C \cup (A \cup B)}(U_{\ell+1}) \big\}\\
&=&\big\{t_{A\cap C}(U_{\ell+1})=t_{A\cup C \cup B}(U_{\ell+1})  \big\}.
\end{eqnarray*}
Since $A$, $B$ and $C$ are measurable for $\F^{X,U}_{\ell}$ and 
$U_{\ell+1}$ is independent from $\F^{X,U}_{\ell}$, one has
$$\P\Big[X_{\ell+1}=Z_{\ell+1};H'~\Big|~\F^{X,U}_{\ell}\Big]
\ge\frac{\mu(A\cap C)}{\mu(A\cup C \cup B)}.$$
Moreover
\begin{eqnarray*}
\mu(A\cap C)&=&\frac{1}{m}\int_E \min\big(f(a|Z_{1:\ell}), m  f(a|X^{\passe}_\ell)\big)\d\pi(a)\\
&=&\frac{1}{m}\Big(\int_E f(a|Z_{1:\ell})\d\pi(a)-\int_E\big[f(a|Z_{1:\ell})- m  f(a|X^\passe_\ell)\big]_+\d\pi(a)\Big)\\
&=&\frac{1}{m}\Big(1-\int_E\big[f(a|Z_{1:\ell})- m  f(a|X^\passe_\ell)\big]_+\d\pi(a)\Big)\\
\end{eqnarray*}
and
\begin{eqnarray*}
\mu(A\cup C \cup B)&=&\int_E \sup\big(m^{-1}f(a|Z_{1:\ell}), f(a|X^\passe_\ell),M\big)\d\pi(a)\\
&=&\int_E\sup(m^{-1}f(a|Z_{1:\ell}),M)\d\pi(a)\\
&\ &\quad\quad\quad+\int_E\big[f(a|X^\passe_\ell)-\sup(m^{-1}f(a|Z_{1:\ell}),M)\big]_+\d\pi(a)\\
&\le&n+1+\int_E\big[f(a|X^\passe_\ell)-M\big]_+\d\pi(a)\\
&\le&(n+1)\Big(1+\int_E\big[f(a|X^\passe_\ell)-M\big]_+\d\pi(a)\Big), 
\end{eqnarray*}
so
\begin{eqnarray*}
\mu(A\cup C \cup B)^{-1}&\ge&\frac{1}{n+1}\Big(1-\int_E\big[f(a|X^\passe_\ell)-M\big]_+\d\pi(a)\Big).
\end{eqnarray*}
Thus
\begin{eqnarray*}
m(n+1)\P\big[X_{\ell+1}=Z_{\ell+1};H'~\big|~\F^{X,U}_{\ell}\big]
&\ge&m(n+1)\frac{\mu(A\cap C)}{\mu(A\cup C \cup B)} \\
&\ge&1-\int_E\big[f(a|X^\passe_\ell)-M\big]_+\d\pi(a)\\
& &-\int_E\big[f(a|Z_{1:\ell})- m  f(a|X^\passe_\ell)\big]_+\d\pi(a).
\end{eqnarray*}
Since $H_\ell \in \F^{X,U}_{\ell}$, one has
\begin{eqnarray*}
m(n+1)\P[ X_{\ell+1}=Z_{\ell+1}\,;\,H' | H_{\ell}]
&\ge&1-\E\Big[\int_E\big[f(a|X^\passe_\ell)-M\big]_+\d\pi(a)\Big|H_{\ell}\Big]\\
&-& \E\Big[\int_E\big[f(a|Z_{1:\ell})- m  f(a|X^\passe_\ell)\big]_+\d\pi(a)\Big|H_{\ell}\Big]\\
&\ge&1-2\varepsilon/3,
\end{eqnarray*}
where the last inequality stands from lemmas~\ref{lemme_amorce_1} 
and~\ref{lemme_amorce_2}.
But
\begin{eqnarray*}
\P[X_{\ell+1}=Z_{\ell+1},H'|H_\ell]&=&\P[H'|H_\ell]\,\P[X_{\ell+1}=Z_{\ell+1}|H'\cap H_\ell]\\
&=&\frac{1}{m(n+1)}\P[X_{\ell+1}=Z_{\ell+1}| H_{\ell+1}].
\end{eqnarray*}
Hence,
$$\P[Z_{\ell+1}=X_{\ell+1}|H_{\ell+1}]\ge 1-2\varepsilon/3.$$
But, since $H_{\ell+1} = H_\ell \cap H'$ and since $H'$ is independent of 
$\F^{X,U}_{\ell}$, the induction hypothesis yields
$$\P[Z_{\ell}\ne X_{1:\ell} | H_{\ell+1}] 
= \P[Z_{\ell}\ne X_{1:\ell} | H_{\ell}] \le \varepsilon/3,$$
so
$$
\P[Z_{1:\ell+1}\ne X_{1:\ell+1} | H_{\ell+1}]\le\P[Z_{\ell}\ne X_{1:\ell} | H_{\ell+1}] 
+ \P[Z_{\ell+1}\ne X_{\ell+1} | H_{\ell+1}]\le\varepsilon,
$$
which ends the proof. \hfill $\square$

\end{proof}

\section{End of the proof of theorem~\ref{theo new}}\label{S_end of the proof}

In this section, we assume that the priming condition and condition $\mathcal{H}'(\eta)$ hold. We show that the parametrization $(U_n)_{n\in \z}$ given by proposition~\ref{governing sequence} is generating. By stationarity, it suffices to show that $X_0$ is a function of $(U_n)_{n\in \z}$ only. We proceed by successive approximations.

\subsection{Approximation until a given time}
\label{S2.3}

Choose $\varepsilon>0$ and $\ell\ge1$ such that $\sum_{n\ge \ell} \eta_n \le \varepsilon$, let $J=[s,t]$ be an interval of integers such that $t-s+1=\ell$. Note $X_J=X_{s:t}$.

By stationarity of $X$,
lemma~\ref{lemme amorce2} provides an event $H_J$ and a random variable $Z_J$, functions of $U_J$ only, such that
\begin{itemize}
        \item $\P\big[H_J]>0$,
        \item $\mathcal{L}(Z_J)=\mathcal{L}(X_J)=\mathcal{L}(X_{1:\ell})$,
        \item $H_J$ is independent of $Z_J$,
        \item $\P\big[X_{J}\ne Z_J~\big|~H_J\big] \le \varepsilon$.
\end{itemize}

Using $Z_J$ and the parametrization $(U_n)_{n \geq t+1}$, we consider the random variables $(X'_n)_{n \geq s}$ defined by $X'_{J}=Z_J$ and for every $n\ge t+1$,
$$X'_n=x_{f'_{n-1}}(U_n) \text{ where } f'_{n-1}= f(\cdot|X'_{s:n-1}).$$

Let us establish some properties of the process $X'$ thus defined.
\begin{lemme}\label{X tilde a meme loi que X}
For every $n \geq s$, the law of $X'_{s:n}$ is the law of $X_{s:n}$.
\end{lemme}

\begin{proof}
Choose $n\geq t+1$, ~$z \in E^{n-s}$ and $B\in \mathcal{E}$. By lemma~\ref{lemme-coupl-loi} and by independence of $U_n$ and $f'_{n-1}$, the random variable $X'_n$ admits the density $f'_{n-1}= f(\cdot|X'_{s:n-1})$ conditionally on $X'_{s:n-1}$. Hence
$$\P\big[X'_n\in B~\big|~X'_{s:n-1}=z\big]~=~\int_B f(x|z)\d\pi(x)~=~\P\big[X_n\in B~\big|~X_{s:n-1}=z\big].$$
Since  $X'_{J}=Z_J$ has the same law as $X_{J}$, the result follows by induction.
\hfill $\square$ \end{proof}

\begin{lemme}\label{lemme maj erreur partie 2} One has
$\P\big[X'\ne X\text{ {\rm on }}[s,+\infty[~\big|~H_J\big]\leq3\varepsilon.$
\end{lemme}

\begin{proof}
Let $n \ge t+1$. Since $X_n=x_{f_{n-1}}(U_n)$ and $X'_n=x_{f'_{n-1}}(U_n)$, proposition~\ref{lemme-coupl-loi} and the independence of $U_n$ and $\F^{X,U}_{n-1}$ yield,
$$\P\big[X'_n\ne X_n~\big|~ \F^{X,U}_{n-1}\big]\leq 2\|f'_{n-1}- f_{n-1}\|.$$
Let $$p_n=\P\big[X'_n\ne X_n~;~X'_{s:n-1}=X_{s:n-1}~;~H_J\big].$$
Since $\big\{X'_{s:n-1}=X_{s:n-1}~;~H_J\big\} \in \F^{X,U}_{n-1},$ one gets
\begin{eqnarray*}
p_n&\leq& \E\Big[2\|f_{n-1}- f'_{n-1}\|~  \mathbf{1}_{\{X'_{s:n-1}=X_{s:n-1}\}}~\mathbf{1}_{H_J}\Big]\\
&=&2\ \E\Big[\|f(\cdot|X^{\passe}_{s-1}X'_{s:n-1})- f(\cdot|X'_{s:n-1})\|~\mathbf{1}_{\{X'_{s:n-1}=X_{s:n-1}\}}~ \mathbf{1}_{H_J}\Big]\\
&\leq& 2\ \E\Big[\|f(\cdot|X^{\passe}_{s-1}X'_{s:n-1})- f(\cdot|X'_{s:n-1})\|~\mathbf{1}_{H_J}\Big].
\end{eqnarray*}

But $X^{\passe}_{s-1}$, $Z_J$, $H_J$ and $U_{t+1:n-1}$ are independent hence $X^{\passe}_{s-1}$, $X'_{s:n-1}$ and $H_J$ are independent since $X'_{s:n-1}$ is a function of $Z_J$ and $U_{t+1:n-1}$ only. Thus,
$$
p_n\leq 2\ \E\Big[\|f(\cdot|X^{\passe}_{s-1}X'_{s:n-1})- f(\cdot|X'_{s:n-1})\| \Big]~ \P[H_J]=2 \eta_{n-s}\P[H_J].
$$
Hence,
$$\P\big[X'_n\ne X_n~;~X'_{s:n-1}=X_{s:n-1}~\big|~H_J\big]\le 2\eta_{n-s},$$
therefore,
$$\P\big[X'_{s:n}\ne X_{s:n}~|~H_J\big]\le \P\big[X'_{s:n-1}\ne X_{s:n-1}~|~H_J\big] + 2\eta_{n-s}.$$
By induction, one gets for all $n \ge t+1$
$$\P\big[X'_{s:n}\ne X_{s:n}~\big|~H_J\big]\le \P\big[X'_{J}\ne X_{J}~\big|~H_J\big]+2\displaystyle\sum_{m=\ell}^{n-s}\eta_m.
$$
Since $X'_{J}=Z_J$ and $\P[X_J\ne Z_J ~|~H_J] \le \varepsilon$, this yields
$$
\P\big[X'\ne X\text{ on }[s,+\infty[~\big|~H_J\big]\le \varepsilon+2\sum_{m=\ell}^{\infty}\eta_m \le 3\varepsilon,$$
which ends the proof. \hfill $\square$ \end{proof}

\subsection{Successive approximations}
\label{S2.4}

Our next step in the proof of theorem~\ref{theo new} is to approach 
the random variable $X_0$ by measurable functions of the 
parametrization $(U_n)_{n\in\z}$ of proposition~\ref{governing sequence}. 
To this aim, we group the innovations by intervals of times.
For every $m \ge 1$ one chooses an integer $L_m$ such that 
$$\sum_{n\geq L_m} \eta_n \leq 1/m.$$
Lemma~\ref{lemme amorce2} (the priming lemma) applied to $\ell=L_m$ and $\varepsilon=1/m$ provides an event $H_{L_m}$ of positive probability $\alpha_{m}$ and a random variable $Z_{L_m}$ such that $$\P \big[X_{1:L_m}=Z_{L_m}~|~H_{L_m}\big]\geq 1-1/m.$$
Choose an integer $M_m\geq 1/\alpha_{m}$.
Split $\z^{*}_-$ into $M_1$ intervals of length $L_1$, $M_2$ intervals of length $L_2$, and so on. More precisely set, for every $n\geq 1$, $\ell_n=L_{m(n)}$, $\varepsilon_n=1/m(n)$ and $\alpha_{n}=\alpha_{m(n)}$ where $m(n)$ is the only integer such that $$M_1+\cdots+M_{m(n)-1}<n\leq M_1+\cdots+M_{m(n)}.$$
Therefore, for every $k\geq0$, 
$$\sum_{n\geq \ell_k}\eta_n \leq \varepsilon_k.$$
For every $k\geq0$, set
$$t_k=-\sum_{1\leq n\leq k}\ell_n,
J_k=[t_k,t_k+\ell_k-1]=[t_k,t_{k-1}-1] \text{ and } X_{J_k}=X_{t_k:t_{k-1}-1}.$$


\begin{figure}[h]
\begin{center}
\includegraphics[width=1\textwidth]{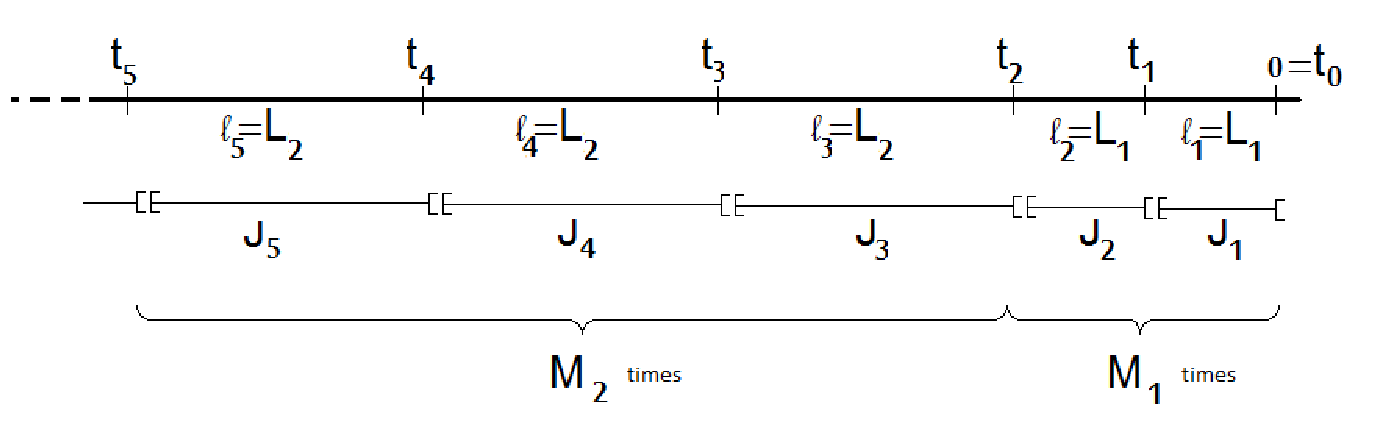}
\caption{Splitting $\z^*_-$ into intervals of times.}
\end{center}
\label{intervals of time}
\end{figure}

Lemma~\ref{lemme amorce2} applied to $(\varepsilon_k)_{k\ge 1}$ and $(U_{J_k})_{k\ge1}$ provides events $(H_{J_k})_{k\ge1}$ and random variables $(Z_{J_k})_{k\ge1}$.
For every $k\geq 0$, let us use the construction of section~\ref{S2.3}: set $X^k_{J_k}=Z_{J_k}$, then for every $n\geq t_k+\ell_k=t_{k-1}$
$$X^k_n=x_{f^k_{n-1}}(U_n) \text{ where } f^k_{n-1}= f(\cdot|X^k_{t_k:n-1}).$$
Therefore, lemma~\ref{lemme maj erreur partie 2} yields the inequality
$$\P[X_0\ne X^k_0~|~H_{J_k}]\le \P\big[X_{t_k:0}\ne X^k_{t_k:0}~\big|~H_{J_k}\big]
\le 3\varepsilon_k,$$
which shows that 
$$\P\big[X_0^k \ne X_0 ~\big|~ H_{J_k}\big] \to 0\text{ when } k \to +\infty.$$
Moreover, the events $H_{J_k}$ are independent as functions of random variables $U_j$ for disjoint sets of indices $j$ and
$$\sum_{k\geq 1} \P\big[H_{J_k} \big]=\sum_{n\geq 1} \alpha_{\ell_{m(n)}} =\sum_{m=1}^{+\infty}M_m \alpha_{m} =+\infty$$
since $M_m \alpha_{m}\ge 1$  by choice of $M_m$.

Lemma~\ref{lem2.1}, stated below, provides a deterministic increasing function $\theta$ such that
$$\sum_{k\geq 1} \P\big[X_0^{\theta(k)}\ne X_0 ~;~ H_{J_{\theta(k)}} \big] < +\infty$$
$$\text{and } \sum_{k\geq 1} \P\big[ H_{J_{\theta(k)}}\big] = +\infty.$$
Using Borel-Cantelli's lemma, one deduces that
\begin{itemize}
\item $\big\{X_0^{\theta(k)}\ne X_0\big\} \cap  H_{J_{\theta(k)}} 
\text{ occurs only for finitely many }k\text{ only, a.s.}$
\item $ H_{J_{\theta(k)}}  \text{ occurs for infinitely many } k \text{ a.s.}$
\end{itemize}
Thus, for every $B \in \mathcal{E}$,
$$\{X_0\in B\} = \limsup_{k \to \infty} 
H_{J_{\theta(k)}} \cap \big\{X_0^{\theta(k)}\in B\big\},$$
hence $\big\{x_0\in B\big\} \in \F^U_0$.
By stationarity of the process $(X,U)$, one gets the inclusion 
of the filtration $\F^X$ into the filtration $\F^U$, which ends the proof. 
$\hfill \square~$

\begin{lemme}\label{lem2.1}
Let $(a_n)_{n\ge0}$ and $(b_n)_{n\ge0}$ denote two bounded sequences of 
nonnegative real numbers such that the series $\sum_nb_n$ diverges 
and such that $a_n \ll b_n$. Then there exists an increasing function 
$\theta : \n \to \n$ such that the series
$\sum_{n} a_{\theta(n)}$ converges and the series
$\sum_{n} b_{\theta(n)}$ diverges.
\end{lemme}

A proof of this lemma can be found in~\cite{Ceillier-splitwords}.

\section{proof of proposition~\ref{global-countable}}\label{S_aux}

Assume that $E$ is countable. Let $\varepsilon=(\varepsilon_a)_{a\in E}$ be 
an i.i.d. family of exponential random variables with parameter 1.

Let us show that, for any probability $p$ on $E$, there exists almost 
surely one and only one $a\in E$ such that
$$\frac{\varepsilon_a}{p(a)} = \inf_{b \in E} \frac{\varepsilon_b}{p(b)} 
\quad (\star),$$
with the convention that $\varepsilon_b/p(b)=+\infty$ if $p(b)=0$.
For every positive real number $r$,
$$\sum_{b \in E}\P[\varepsilon_b/p(b)\le r]
= \sum_{b \in E}\big(1-e^{-p(b)r}\big)\le \sum_{b\in E} p(b)r =r < +\infty,$$
hence Borel-Cantelli's lemma ensures that the event 
$\{\varepsilon_b/p(b)\le r\}$ occurs only for finitely many $b \in E$.
Thus the infimum in $(\star)$ is achieved at some $a \in E$. 
The uniqueness follows from the equalities 
$\P[\varepsilon_a/p(a)=\varepsilon_b/p(b)< +\infty]=0$ for every $a \ne b$, 
since $\varepsilon_a$ and $\varepsilon_b$ are independent random 
variables with diffuse laws.

Define $g(\varepsilon,p)$ (almost surely) as the only index $a$ such  
$(\star)$ holds. Let us show that, for every $p,q \in  \mathcal{P}(E)$,
$$
\P[g(\varepsilon,p) \ne g(\varepsilon,q)] \leq 2\frac{\|p-q\|}{1+\|p-q\|}.
$$
For every $a,b \in E$, set 
$C_a= \{g(\varepsilon,p) = g(\varepsilon,q) = a\}$ and
$$\lambda_{b/a} = \max \Big(\frac{p(b)}{p(a)},\frac{q(b)}{q(a)}\Big).$$
Fix $a\in E$. Then up to negligible events,
$$\{g(\varepsilon,p)=a\} 
= \bigcap_{b\ne a}
\Big\{\frac{\varepsilon_a}{p(a)}\le \frac{\varepsilon_b}{p(b)}\Big\}
= \bigcap_{b\ne a}
\Big\{\varepsilon_b \ge \frac{p(b)}{p(a)} \varepsilon_a \Big\},$$
and the same holds for $q$. Thus,
\begin{eqnarray*}
C_a = \bigcap_{b\ne a}\left\{ \varepsilon_b \ge\lambda_{b/a}\varepsilon_a\right\},
\end{eqnarray*}
Conditioning on $\varepsilon_a$ and using the fact that the random 
variables $(\varepsilon_b)$ are i.i.d. and exponentially distributed, 
one gets
$$
\P\big[C_a\,\big|\,\varepsilon_a\big]
=
\prod_{b \ne a} \exp\left(- \lambda_{b/a} \varepsilon_a\right),
$$
hence
\begin{eqnarray}\label{C_a}
\P(C_a)
= \E\Big[ \exp \Big( -\sum_{b \ne a} \lambda_{b/a} \varepsilon_a \Big) \Big]
= \Big( 1 + \sum_{b \ne a} \lambda_{b/a} \Big)^{-1}
= \Big( \sum_{b} \lambda_{b/a} \Big)^{-1}.
\end{eqnarray}
But $\lambda_{b/a}\le \max(p(b),q(b))/\min(p(a),q(a))$, hence
\begin{eqnarray*}
\P(C_a) \ge \frac{\min(p(a),q(a))}{\sum_{b} \max(p(b),q(b))}
= \frac{\min(p(a),q(a))}{1+\|p-q\|}.
\end{eqnarray*}
Therefore
\begin{eqnarray*}
\P[g(\varepsilon,p) = g(\varepsilon,q)]
=\sum_{a \in E}\P(C_a)
=\frac{1-\|p-q\|}{1+\|p-q\|}.
\end{eqnarray*}
Last, note that if $p=q$, then for each $a \in E$, equation~(\ref{C_a}) becomes $$\P[g(\varepsilon,p)=a]=\P(C_a)=\Big(\sum_b\frac{p(b)}{p(a)}\Big)^{-1}=p(a),$$ which shows that
the law of $g(\varepsilon,p)$ is $p$. The proof is complete.
\hfill $\square$


\end{document}